\documentclass[leqno]{article}
\usepackage[pdftex]{hyperref}
\usepackage{verbatim}
\usepackage{amsmath}
\usepackage{amsthm}
\usepackage{amssymb,latexsym}
\usepackage[dvips]{graphicx}
\usepackage{cancel}

\theoremstyle{definition}
\newtheorem{theorem}{Theorem}[section]
\newtheorem{lemma}[theorem]{Lemma}
\newtheorem{corolary}[theorem]{Corolary}
\newtheorem{proposition}[theorem]{Proposition}

\begin{document}
\title{The Cantor's First Diagonal Formalized and Extended}
\author{Jo\~{a}o Alves Silva J\'{u}nior}
\date{\today}

\maketitle

\section{Introduction}\label{S:Intro}

Diagrams like that in Figure~\ref{F:CantDiag} are often presented as proof that $\mathbb{N}\times\mathbb{N}$ and $\mathbb{N}$ have the same cardinality (see \cite[p.~8]{BBJ07}, \cite[p.~76]{HJ99} and \cite[p.~10]{YM07}). The arrows indicate the growth direction of a function $\Psi: \mathbb{N}\times\mathbb{N} \to \mathbb{N}$ that intuitively covers, without repetitions, all the elements of $\mathbb{N}$. So, $\Psi$ is a bijection between $\mathbb{N}\times\mathbb{N}$ and $\mathbb{N}$. This is the idea known as Cantor's first diagonal. None of these texts formally defined such function $\Psi$, neither rigorously proved that $\Psi$ is a bijection between $\mathbb{N}\times\mathbb{N}$ and $\mathbb{N}$. The purpose of this paper is to fill these gaps in a more general context, finding a simple closed-form expression for a bijection between $\mathbb{N}^k$ and $\mathbb{N}$, where $k$ is an arbitrary positive integer.

\begin{figure}[here]\label{F:CantDiag}
\begin{center}
\includegraphics[scale=1]{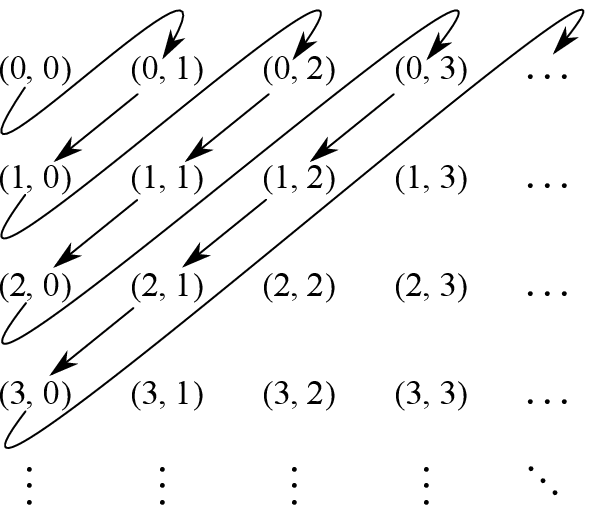}
\end{center}
\caption{Cantor's first diagonal.}
\end{figure}

Given $r \in \mathbb{N}$, let $P_r$ be the set $\{ i \in \mathbb{N} \mid 1 \leq i \leq r\}$. We denote arbitrary elements of $\mathbb{N}^k$ using bold letters like $\boldsymbol{m}$, $\boldsymbol{n}$ and $\boldsymbol{p}$. The $i$-th coordinate of a $\boldsymbol{m} \in \mathbb{N}^k$ is denoted by $m_i$. So, $\boldsymbol{m} = \boldsymbol{n}$ if and only if $m_i = n_i$ for all $i \in P_k$.

Let $\mathrm{Diff}(\boldsymbol{m}, \boldsymbol{n})$ denote the set of all indices $i \in \mathbb{N}-\{0\}$ such that $m_i$ and $n_i$ are defined and $m_i \neq n_i$, for all $\boldsymbol{m}, \boldsymbol{n} \in \bigcup_{r = 1}^\infty \mathbb{N}^r$. A natural way of ordering elements of $\mathbb{N}^k$ is given by the \emph{lexicografic order} $<_k$, defined by
\begin{equation}\label{E:DefLexOrd}
\boldsymbol{m} <_k \boldsymbol{n} \quad \Leftrightarrow \quad \exists i \in P_k, \ m_i < n_i \text{ and } i = \min \mathrm{Diff}(\boldsymbol{m}, \boldsymbol{n}).
\end{equation}
We are specially interested in considering this relation on the subset
\begin{equation}\label{E:DefDk}
 \mathcal{D}_k = \{ \boldsymbol{n} \in \mathbb{N}^k \mid n_1 \geq \dots \geq n_k \},
\end{equation}
which is clearly equipotent to $\mathbb{N}^k$ via
\begin{align}\label{E:Defhk}
h_k: \mathcal{D}_k &\to \mathbb{N}^k\\
  \boldsymbol{n} &\mapsto (n_k, n_{k-1} - n_k, n_{k-2} - n_{k-1}, \dots, n_1 - n_2) \notag.
\end{align}
Note in Figure \ref{F:CantDiag} that there is an arrow from $(m_1, m_2)$ to $(n_1, n_2)$ if and only if $h_2^{-1}(m_1, m_2)$ immediately precedes $h_2^{-1}(n_1, n_2)$ in $(\mathcal{D}_2, <_2)$. Our approach is based on a generalization of this idea for $\mathbb{N}^k$.

\section{Two inverse bijections between $\mathbb{N}^k$ and $\mathbb{N}$}

\begin{proposition}\label{T:LexOrd}
$(\mathcal{D}_k, <_k)$ is a well-ordered set, without maximum, such that every nonempty subset of $\mathcal{D}_k$ with an upper bound has a $<_k$-maximum.
\end{proposition}
\begin{proof}
See \cite[p.~82]{HJ99} for a proof that $<_k$ is irreflexive, transitive and total. The induced order $<_k \cap \,(\mathcal{D}_k \times \mathcal{D}_k)$ inherits these properties. Given a nonempty $A \subseteq \mathcal{D}_k$, define $\boldsymbol{m}$ by $m_i = \min \{ n_i \mid \boldsymbol{n} \in A_{i-1}\}$, where $A_0 = A$ and $A_i = \{ \boldsymbol{n} \in A_{i-1} \mid n_i = m_i\}$, for all $i \in P_k$. It is an easy exercise to verify that $\boldsymbol{m}$ is well-defined and it is the $<_k$-minimum of $A$. The maximum of a nonempty $A \subseteq \mathcal{D}_k$ with an upper bound is determined analogously (just replace min with max). Finally, $(\mathcal{D}_k, <_k)$ doesn't have a maximum because for all $\boldsymbol{m} \in \mathcal{D}_k$, there is a $\boldsymbol{n} = (m_1+1, 0 \dots, 0) \in \mathcal{D}_k$ such that $\boldsymbol{m} <_k \boldsymbol{n}$.
\end{proof}

Supported by Proposition \ref{T:LexOrd}, we define a function $f_k: \mathbb{N} \to \mathcal{D}_k$ by
\begin{equation}\label{E:Deffk}
\left\{ \begin{aligned} f_k(0) &= \min \mathcal{D}_k \\ \forall x\in \mathbb{N}, \ f_k(x+1) &= \min \{ \boldsymbol{n} \in \mathcal{D}_k \mid f_k(x) <_k \boldsymbol{n} \}.\end{aligned}\right.
\end{equation}

\begin{proposition}\label{T:fkBijection}
$f_k$ is one-to-one onto $\mathcal{D}_k$.
\end{proposition}
\begin{proof}
Use Proposition \ref{T:LexOrd} mimicking the proof of Theorem~3.4 of \cite{HJ99}.
\end{proof}

Let $\Phi_k: \mathcal{D}_k \to \mathbb{N}$ be defined by
\begin{equation}\label{E:DefPhik}
\Phi_k(\boldsymbol{n}) = \sum_{i=1}^{k}\binom{k-i+n_i}{k-i+1},
\end{equation}
where the binomial coefficients $\binom{x}{y}$ are generalized for $x,y \in \mathbb{Z}$ (see \cite[\S 2.3.2]{RMG00}):
\[ \binom{x}{y} = \begin{cases} (-1)^y\binom{y-x-1}{y} &\text{if $y\geq0$ and $x < 0$,}\\ 0 &\text{if $y < 0$ or $x<y$.} \end{cases} \]
We now prove that $\Phi_k$ is the inverse function of $f_k$.

\begin{lemma}\label{T:fkLemma}
Given $\boldsymbol{m} \in \mathcal{D}_k$ and $x \in \mathbb{N}$, suppose that $f_k(x) = \boldsymbol{m}$. Let $m_0$ denote $m_1 + 1$. If $r \in P_k$ is such that $m_{r-1} > m_{r}$ and $m_i = m_{r}$ for all $i \in \{r, \dots, k\}$, then $f_k(x+1) = (m_1, \dots, m_{r-1}, m_{r}+1, 0, \dots, 0)$. In particular, when the coordinates of $\boldsymbol{m}$ are all equal, $f_k(x+1) = (m_1+1, 0, \dots, 0)$.
\end{lemma}
\begin{proof}
Let $A$ be the set $\{ \boldsymbol{n} \in \mathcal{D}_k \mid \boldsymbol{m} <_k \boldsymbol{n} \}$, so that $f_k(x+1) = \min A$. We want to prove that $\min A = \boldsymbol{p}$, where $\boldsymbol{p} = (m_1, \dots, m_{r-1}, m_{r}+1, 0, \dots, 0)$. It is easy to note that $\boldsymbol{p} \in A$. So, since $<_k$ is a \emph{linear} order, it remains only to show that $\boldsymbol{p}$ is minimal. Suppose towards a contradiction that $\boldsymbol{n} <_k \boldsymbol{p}$, for some $\boldsymbol{n} \in A$. Let $j = \min \mathrm{Diff}(\boldsymbol{n},\boldsymbol{p})$, so that $n_j < p_j$. Since $p_{r+1} = \dots = p_k = 0$, $j \leq r$. But $j \geq r$, otherwise we would have $\min \mathrm{Diff}(\boldsymbol{m},\boldsymbol{n}) = j$ and, hence, $m_j < n_j < p_j = m_j$ (because $m_i = p_i$ for all $i \in \{1, \dots, r-1\}$). So, $r = j$, whence $n_r < p_r = m_r +1$ and $n_i = p_i = m_i$, for all $i \in \{1, \dots, r-1\}$. Since $\boldsymbol{n} \in \mathcal{D}_k$ and $m_r = \dots = m_k$, it follows that $n_i \geq m_i$ for all $i \in P_k$ (see Figure \ref{F:mnDiagram}). But this contradicts the hypothesis that $\boldsymbol{m} <_k \boldsymbol{n}$. Thus, $\boldsymbol{p}$ is minimal.
\end{proof}

\begin{figure}[here]
\centering
\includegraphics[scale=1]{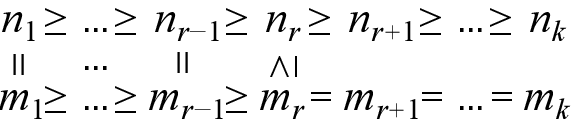}
\caption{$m_i \geq n_i$, for all $i \in P_k$.}
\label{F:mnDiagram}
\end{figure}

\begin{lemma} \label{T:PhikCircfkEqualsId}
$(\Phi_k \circ f_k)(0) = 0$ and $(\Phi_k \circ f_k)(x+1) = (\Phi_k \circ f_k)(x) + 1$, for all $x \in \mathbb{N}$. Thus, $\Phi_k \circ f_k$ is $\mathrm{id}_{\mathbb{N}}$, the identity function on $\mathbb{N}$.
\end{lemma}
\begin{proof}
$(0, \dots, 0)$ is clearly the minimum of $(\mathcal{D}_k, <_k)$ (in fact, it is the minimum of $(\mathbb{N}^k, <_k)$) and $\Phi_k(0, \dots, 0) = 0$. Hence, $(\Phi_k \circ f_k)(0) = 0$.

Given $x \in \mathbb{N}$, let $\boldsymbol{m}$, $\boldsymbol{p}$ and $m_0$ denote $f_k(x)$, $f_k(x+1)$ and $m_1 + 1$, respectively. There is an unique $r \in P_k$ such that $m_{r-1} > m_{r}$ and $m_i = m_{r}$ for all $i \in \{r, \dots, k\}$. By Lemma \ref{T:fkLemma}, $\boldsymbol{p} = (m_1, \dots, m_{r-1}, m_{r}+1, 0, \dots, 0)$. So, by \eqref{E:DefPhik},
\begin{align}
\Phi_k(\boldsymbol{p}) &= \sum_{i = 1}^{r-1} \binom{k-i+m_i}{k-i+1} + \binom{k-r+m_r+1}{k-r+1}, \label{E:Phikp} \\
\Phi_k(\boldsymbol{m}) &= \sum_{i = 1}^{r-1} \binom{k-i+m_i}{k-i+1} + \sum_{i=r}^k\binom{k-i+m_r}{k-i+1}. \label{E:Phikm}
\end{align}
On the other hand, by the parallel summation identity (see \cite[\S 2.3.4]{RMG00}),
\begin{equation}\label{E:ParallelSummation}
 \forall x,y \in \mathbb{Z}, \quad \sum_{i=0}^{y}\binom{i+x-1}{i} = \binom{x+y}{y}.
\end{equation}
Applying \eqref{E:ParallelSummation} for $x = m_r$ and $y = k-r+1$, we obtain an equation which can be modified by a changing of index (namely, replacing $i$ with $k-i+1$) into
\begin{equation}\label{E:AuxForPhikpEqualsPhikmPlus1}
 \binom{m_r+k-r+1}{k-r+1} = \sum_{i=r}^{k} \binom{k-i+m_r}{k-i+1} + 1.
\end{equation}
By \eqref{E:Phikp}, \eqref{E:Phikm} and \eqref{E:AuxForPhikpEqualsPhikmPlus1}, it results that $\Phi_k(\boldsymbol{p}) = \Phi_k(\boldsymbol{m}) +1$. That is, $(\Phi_k \circ f_k)(x+1) = (\Phi_k \circ f_k)(x) + 1$.

The last assertion follows from the uniqueness part of the recursion theorem (see \cite[p.~53]{YM07}), since $\mathrm{id}_{\mathbb{N}}(0) = 0$ and $\mathrm{id}_{\mathbb{N}}(x+1) = \mathrm{id}_{\mathbb{N}}(x)+1$ for all $x \in \mathbb{N}$.
\end{proof}

\begin{proposition}
$\Phi_k$ is the inverse function of $f_k$.
\end{proposition}
\begin{proof}
By Proposition \ref{T:fkBijection}, $f_k$ is invertible. So, by Lemma \ref{T:PhikCircfkEqualsId}, $\Phi_k = \Phi_k \circ \mathrm{id}_{\mathbb{N}} = \Phi_k \circ (f_k \circ f_k^{-1}) = (\Phi_k \circ f_k) \circ f_k^{-1} =  \mathrm{id}_{\mathbb{N}} \circ  f_k^{-1} = f_k^{-1}$.
\end{proof}

\begin{corolary}
$h_k \circ f_k: \mathbb{N} \to \mathbb{N}^k$ has an inverse $\Psi_k: \mathbb{N}^k \to \mathbb{N}$ given by
\begin{equation}\label{E:DefPsik}
\Psi_k(\boldsymbol{n}) = \sum_{i=1}^{k}\binom{i-1+n_1+ \dots + n_i}{i}.
\end{equation}
\end{corolary}
\begin{proof}
For all $\boldsymbol{n} \in \mathbb{N}^k$, the $i$-th coordinate of $h_k^{-1}(\boldsymbol{n})$ is $\sum_{j=1}^{k-i+1}n_j$. Moreover, since $h_k$ and $f_k$ are invertible and $f_k^{-1} = \Phi_k$, $\Psi_k = \Phi_k \circ h_k^{-1}$ is the inverse of $h_k \circ f_k$. So, $h_k \circ f_k$ has an inverse function $\Psi_k$ given by
\[ \Psi_k(\boldsymbol{n}) = \sum_{i=1}^{k}\binom{k-i+\sum_{j=1}^{k-i+1}n_j}{k-i+1} = \sum_{i=1}^{k}\binom{i-1+n_1+ \dots + n_i}{i}.  \]
The last equality is obtained by reversing the order of the summands.
\end{proof}

\section{Combinatorial remarks}

So far, we have not explained where the formula \eqref{E:DefPhik} came from. It can be deduced combinatorially after noticing that $f_k^{-1}(\boldsymbol{n})$ is the cardinality of $\mathcal{D}_k[\boldsymbol{n}] = \{\boldsymbol{m} \in \mathcal{D}_k \mid \boldsymbol{m} <_k \boldsymbol{n} \}$, which is the disjoint union of the family $\{\mathcal{D}_k[\boldsymbol{n};i]\}_{i\in P_k}$, where $\mathcal{D}_k[\boldsymbol{n};i] = \{ \boldsymbol{m} \in \mathcal{D}_k \mid m_i < n_i \text{ and } i = \min \mathrm{Diff}(\boldsymbol{m},\boldsymbol{n})\}$. Given $i \in P_k$, the mapping $g_k[\boldsymbol{n};i]: \boldsymbol{m} \mapsto (m_i, \dots, m_k)$, defined on $\mathcal{D}_k[\boldsymbol{n};i]$, is one-to-one onto $\mathcal{D}_{k-i+1}$ and $h_{k-i+1} \circ g_k[\boldsymbol{n};i]$ sends $\mathcal{D}_k[\boldsymbol{n};i]$ biunivocally to
\begin{align}
\mathrm{Im}\,(h_{k-i+1} \circ g_k[\boldsymbol{n};i]) &= \left\{\boldsymbol{p} \in \mathbb{N}^{k-i+1} \mid p_1 + \dots + p_{k-i+1} < n_i \right\} \\
&= \bigsqcup_{j = 0}^{n_i-1}\left\{\boldsymbol{p} \in \mathbb{N}^{k-i+1} \mid p_1 + \dots + p_{k-i+1} = j \right\},\notag
\end{align}
where $\bigsqcup$ denotes disjoint union. But, according \cite[\S 2.3.3]{RMG00},
\begin{equation}
\forall j \in \mathbb{Z}, \quad |\left\{\boldsymbol{p} \in \mathbb{N}^{k-i+1} \mid p_1 + \dots + p_{k-i+1} = j \right\}| = \binom{k-i+j}{j}.
\end{equation}
So, by \eqref{E:ParallelSummation},
\begin{equation}
|\mathrm{Im}\,(h_{k-i+1} \circ g_k[\boldsymbol{n};i])| = \sum_{j = 0}^{n_i-1} \binom{k-i+j}{j} = \binom{k-i+n_i}{k-i+1}.
\end{equation}
Now, by the foregoing considerations,
\begin{equation}
|\mathcal{D}_k[\boldsymbol{n}]| = \sum_{i=1}^{k}|\mathcal{D}_k[\boldsymbol{n};i]| \quad \text{ and } \quad |\mathrm{Im}\,(h_{k-r+1} \circ g_k[\boldsymbol{n};r])| = |\mathcal{D}_k[\boldsymbol{n};r]|,
\end{equation}
for all $r \in P_k$. Thus,
\begin{equation}
 |\mathcal{D}_k[\boldsymbol{n}]| = \sum_{i=1}^{k}\binom{k-i+n_i}{k-i+1}.
\end{equation}

To finish with, we prove a theorem that confirms that $\boldsymbol{n} \mapsto |\mathcal{D}_k[\boldsymbol{n}]|$ is a bijection between $\mathcal{D}_k$ and $\mathbb{N}$. It may be useful for constructing other explicitly defined injections onto $\mathbb{N}$.

\begin{theorem}\label{T:Criterion}
Let $(A,\prec)$ be a linearly ordered set. If $A$ is infinite and, for all $a \in A$, the set $\{ x \in A \,|\, x \prec a \}$ is finite, then the function $\Psi : A \to \mathbb{N}$ given by $\Psi(a) = |\{x \in A \, | \, x \prec a\}|$ is a bijection between $A$ and $\mathbb{N}$.
\end{theorem}
\begin{proof}
Let $A[a]$ denote $\{x \in A \, | \, x \prec a \}$, for all $a \in A$. Given $a,b,x \in A$, with $a \prec b$, $x \prec a$ implies $x \prec b$ and $x \nprec a$ (because $\prec$ is irreflexive and transitive), so that $A[a]$ is strictly contained in $A[b]$. Since $A[a]$ and $A[b]$ are finite, it follows that  $\Psi(a) = |A[a]| < |A[b]| = \Psi(b)$. Thus, $\Psi$ is strictly growing. Since $\prec$ is total on $A$, it is easy to note that the strict growing of $\Psi$ ensures its injectivity. It remains to prove that $\mathrm{Im}\,\Psi = \mathbb{N}$, where $\mathrm{Im}\,\Psi$ denotes the image of $\Psi$.

Suppose, to get a contradiction, that some $r \in \mathbb{N}$ is not in $\mathrm{Im}\, \Psi$. Since $A$ is infinite and linearly ordered by $\prec$, there are $a_0,\dots,a_r\in A$, pairwise distinct, such that $a_0\prec\dots\prec a_r \, \therefore \, \Psi(a_0)<\dots<\Psi(a_r) \, \therefore \, r<\Psi(a_r)$. So, the set $B = \{y \in \mathrm{Im}\, \Psi \, | \, r < y \}$ is nonempty. Let $w = \min B$ and $\alpha = \Psi^{-1}(w)$. Since $0 \leq r < w$, $|A[\alpha]| = \Psi(\alpha) = w > 0$, whence $A[\alpha]$ is finite nonempty. Thus, exists $\beta = \max A[\alpha]$ such that $A[\alpha] = A[\beta]\cup\{\beta\}$ and $A[\beta]\cap\{\beta\} = \varnothing$. Then, $w = \Psi(\alpha) = |A[\alpha]| = |A[\beta]| + 1 = \Psi(\beta) + 1 \, \therefore \, w-1 = \Psi(\beta) \in \mathrm{Im}\, \Psi$. Moreover, $r \leq w - 1$, because $w \in B \therefore r < w$. But $r \notin \mathrm{Im}\, \Psi \ni w-1$, so that $r \neq w-1$. Hence, $r < w-1 \, \therefore \, w-1 \in B$, contradicting the minimality of $w$.
\end{proof}

\end{document}